\documentclass[12pt,leqno,twoside,a4paper]{article}
\usepackage{amsfonts,amsmath,amsthm,amssymb}
\usepackage[english]{babel}
\usepackage{color}
 \usepackage{a4wide}

\newtheorem{theorem}{Theorem}[section]
\newtheorem{prop}[theorem]{Proposition}
\newtheorem{lemma}[theorem]{Lemma}
\newtheorem{corollary}[theorem]{Corollary}
\theoremstyle{definition}
\newtheorem{definition}[theorem]{Definition}
\theoremstyle{definitions}

\newtheorem{example}[theorem]{Example}

\theoremstyle{remark}
\newtheorem{remark}[theorem]{\bf Remark}

 \numberwithin{equation}{section}

 \newcommand{\an}[2]{\mbox{\rm  ann}_{#1}({#2})}
\begin{document}

\title{Conjugate (nil) clean rings and K\"{o}the's problem  }
\author{  Jerzy Matczuk\thanks{This research was supported by the Polish National Center of Science
 Grant No. DEC-2011/03/B/ST1/04893. }\\
    $  $Institute of Mathematics, Warsaw University \\
  Banacha 2, 02-097 Warsaw, Poland\\
  E-mail: jmatczuk@mimuw.edu.pl}
\date{}
\maketitle
\begin{abstract}
  Question 3  of \cite{D} asks  whether the matrix ring $M_n(R)$ is   nil clean, for any   nil clean ring $R$. It is shown that positive answer to this question  is equivalent to positive solution for K\"{o}the's problem in the class of algebras over the field $\mathbb{F}_2$. Other equivalent problems are also discussed.

 The classes of conjugate     clean and conjugate  nil  clean rings, which
lie  strictly between uniquely (nil) clean  and (nil) clean rings are introduced and investigated.
\end{abstract}

The notion of   clean rings was introduced in 1977 by Nicholson  in \cite{N1}. Thereafter such rings and their variations  were intensively studied by many authors (cf. \cite{NZ2} and references within).

Recall that an element $a$ of a unital ring $R$ is clean if $a=e+u$, where $e$ is an idempotent and $u $ is a unit of $R$. When the above presentation is unique, $a$ is called uniquely clean. The ring $R$ is (uniquely) clean if   every element of $R$ is such.

Diesl  in \cite{D}  undertook  to    develop a general theory, based on idempotents and decomposition
of elements, that would unify some of existing concepts related to cleanness and regularity.  In this context a  class of (uniquely) nil clean rings, appeared naturally, i.e. rings in which every element can be (uniquely) presented as  $e+l$, for some idempotent $e$ and a nilpotent element $l$.  It is easy  to see that if $a$ is a  nil clean element, then  $-a$ is clean. Thus  nil clean rings are  clean.
Earlier, uniquely nil clean rings were considered by   Chen in \cite{Ch}.

In the paper, we introduce and investigate conjugate  clean and conjugate nil clean  rings. Those are (nil) clean rings in which idempotents appearing in decompositions  of elements described above are unique up to conjugation, i.e. if $a=e+s=f+t$ are such decompositions, then the idempotents $e,f$ are conjugate in $R$. Clearly every uniquely (nil) clean ring is  conjugate (nil) clean. In fact,  it   is not difficult to see that a ring $R$ is uniquely (nil) clean if and only if it is conjugate (nil) clean and all idempotents of $R$ are central. Thus the introduced classes of clean rings seem to be natural extensions of their "unique" counterparts.   We offer constructions and  characterizations  of such rings. In particular, it will become clear that the introduced classes are different and  the class of conjugate (nil) clean rings lies strictly between uniquely (nil) clean rings and (nil) clean rings. All this is presented in Section 2.

It is known that the matrix ring $ M_n(R)$ over a clean ring $R$ is    also clean (cf. Corollary 1 \cite{HN}). On the other hand,  Wang and  Chen \cite{WC}  constructed a commutative clean ring $R$ such that   not every element of    $M_n(R)$ can be presented as a sum of an idempotent and a unit that commute  with each other. In other words,  they proved that  a matrix ring over strongly clean ring does not have to be strongly clean.

Let $R$ be a nil clean ring. Then, by the above, $M_n(R)$ is a clean ring. Diesl  posed a question (Question 3 \cite{D}) whether  $M_n(R)$ is in fact nil clean. This question was the initial motivation for our studies. It remains unsolved, nevertheless we show that positive answer to this question is equivalent to positive solution for  K\"{o}the's problem in the class of algebras over the field $\mathbb{F}_2=\mathbb{Z}/2\mathbb{Z}$. In fact we present in Theorem \ref{koethe} various conditions related to clean rings which are equivalent to K\"{o}the's problem. It appears that formally weaker statement "$M_2(R)$ is nil clean for any uniquely nil clean $\mathbb{F}_2$-algebra  $R$" is, in fact, equivalent to Diesl 's question. On the other hand, there exist conjugate nil clean rings $R$ such that the matrix ring $M_2(R)$ is not conjugate nil clean.

\section{Preliminary results}

For a ring $R$, $J(R)$ will denote the Jacobson radical of $R$, $U(R)$ will stand for the group  of units of $R$.

The following  proposition will be crucial for our considerations.
\begin{prop}[Corollary 11 \cite{KLM}] \label{Cor idempotents}
 Let $e,f\in R$ be   idempotents such that $e-f$ is a nilpotent element or $e-f\in J(R)$. Then  $e$ and $f$ are conjugate in $R$, i.e. there exists   $u\in U(R)$ such that $e=ufu^{-1}$.
\end{prop}

Let us present an application of the above proposition.  It    will be needed later in the text but it is also of independent interest.  In the following theorem  $T$ is an over ring of a ring $R$ such that $T=R\oplus I$, for some ideal $I$ of $T$. The two-sided annihilator  of $I$ in $R$ is define as $\an{R}{I}=\{ r\in R \mid rI=Ir=0\}$. Recall that a ring $R$ is called abelian if all its idempotents are central.

\begin{theorem}\label{thm1 central idempotents}
  Suppose $T=R\oplus I$, where $I$ is an ideal of $T$ such that $J(I)=I$. The following conditions are equivalent:
 \begin{enumerate}
 \item[(1)]    $T$ is an abelian ring;

\item[(2)] All idempotents of $R$ are central in $T$;

\end{enumerate}
If one of the above equivalent conditions holds then:
\begin{enumerate}
   \item[(3)]  All idempotents of $T$ are trivial (i.e. all idempotents of $T$ belong  to $R$) and $es=se$, for every idempotent $e$ of $R$ and $s\in I$;

\end{enumerate} Moreover,  when $\an{R}{I}=0$,  all the above statements are equivalent to:
\begin{enumerate}
  \item[(4)]  All idempotents of $R$  commute with elements of $I$.
\end{enumerate}

\end{theorem}
\begin{proof} The implications   $(1)\Rightarrow (2)$ and  $(3)\Rightarrow (4)$ are tautologies.

$(2)\Rightarrow (1)\mbox{ and } (3)$. Let $e=e_0+s$ be an idempotent, where $e_0\in R$ and $s\in I$. Then $e ^2=e_0^2+w$, for some $w\in I$.  Thus $e_0\in R$ is an idempotent and the statement (2) implies that $e_0$ is central in $T$.   Moreover, as $e-e_0=s\in I\subseteq J(T)$, we can apply   Proposition \ref{Cor idempotents} to pick   $u\in U(T)$ such that  $e=ue_0u^{-1}=e_0\in R$. This shows that $e=e_0$ belongs to $R$ and it is central in $T$, i.e. statements (1) and (3) hold.

Suppose now that  $\an{R}{I}=0$  and   the property (4) is satisfied.
 Let   $e$ be an idempotent of $R$,    $r\in R$ and $s\in I$. Then, making use of  (4),  we have $0=e(rs)(1-e)=er(1-e)s$ and $0=e( sr)(1-e)=ser(1-e)$. This shows that $er(1-e)I=0=Ier(1-e)$, i.e.   $er(1-e)\in\an{R}{I}=0$.   Replacing the idempotent  $e$ by $(1-e)$ we also obtain $(1-e)re=0$, for any $r\in R$. Hence $e$ is central in $R$ and, by (4), $e$ is central in $T$ i.e. the statement (2) holds. This completes the proof of the theorem.
\end{proof}

Notice that  implications $(1)\Leftrightarrow (2)\Rightarrow (3)\Rightarrow (4)$  in the above theorem  hold  always without   any additional assumptions. Clearly the  equivalence of all the conditions do not hold in general. For example, it is easy to construct rings $T=R\oplus I$, such that    $\an{R}{I}$ contains a noncentral idempotent of $R$ but all idempotents of $R$ commute with elements of $I$ ( taking $T=R$ and $I=0$ we get  a trivial  example of this kind).

 Let $\sigma$ be an endomorphism of a ring $R$ and $ R[x;\sigma], R[[x;\sigma]]$ denote skew polynomial and skew power series rings over $R$, respectively.

\begin{corollary} \label{idempotents in power series} Let $T$ denote one of the rings  $R[[x;\sigma]]$ or $R[x;\sigma]/(x^n)$, where $\sigma$ is an endomorphism of $R$ and $n\geq 2$.  The following conditions are equivalent:
 \begin{enumerate}
 \item[(1)]  $T$ is an abelian ring;

\item[(2)] All idempotents of $R$ are central and  $\sigma(e)=e$, for every idempotent $e$ of $R$;
 \item[(3)]  All idempotents of $T$ are trivial and $\sigma(e)=e$, for every idempotent $e$ of $R$.
\end{enumerate}
\begin{proof} Notice that if $T=R[[x;\sigma]]$ then $I=Tx$ is the Jacobson radical ideal of $T$ and  $T=R\oplus I$.

For $T=R[x;\sigma]/(x^n)$, let   $\bar x$ denote the canonical image of $x$ in $T$. Then $I=T\bar x$ is a nilpotent ideal of $R$ and $T=R\oplus I$.

In both cases  $\an{R}{I}=0$ (in fact the left annihilator of $I$ in $R$ is equal to zero).  It is also standard to check that   an element $a\in R$ commutes with  elements of $I$ if and only if $a$ is central in $R$ and $\sigma(a)=a$. Now it is clear that the corollary is a consequence of Theorem \ref{thm1 central idempotents}.
\end{proof}
\end{corollary}

 We will say that  $a=e+t$   is a   clean  (nil clean)  decomposition of an element $a$ of $ R$ if   $e=e^2$ and $t$ is a unit (a nilpotent element) of $R$.
 \begin{definition}

\begin{enumerate}  An element $a\in R$ is called
  \item[(i)]    conjugate     clean  if it is clean and  for any  two clean  decompositions $a=e+u=f+v$  of $a$,  the idempo\-tents $e,f$ are conjugate;
  \item[(ii)]    conjugate  nil clean  if it is nil clean and  for any two  nil clean  decompositions $a=e+l=f+m$  of $a$,  the idempotents $e,f$ are conjugate.
\end{enumerate}

 \end{definition}
Clearly every uniquely (nil) clean   element is conjugate (nil) clean. Let us observe that:
\begin{remark}

  (1)    Proposition \ref{Cor idempotents} implies that      every idempotent of  a nil clean ring $R$ is conjugate nil clean.\\
    (2)   Let $a$ be a conjugate clean element of $R$. If either $a$  or $a-1$ is invertible   then $a$ is uniquely clean. In particular,
    all nilpotent  elements and all units   which are conjugate   clean   are, in fact,  uniquely   clean.

\end{remark}

Let us recall that idempotents
lift   modulo an ideal $I$ of $R$ if, for any $a\in R$ such that $a^2-a\in I$, there exists an idempotent $e\in R$ such that
$e-a\in I$. If the idempotent $e$ is uniquely determined by the element $a$, then we say that idempotents lift uniquely modulo $I$.

The following lemma is   known (cf. \cite{NZ}). We present its short proof for completeness.
\begin{lemma}[Lemma 17 \cite{NZ}] \label{lemma 17 NZ}
Let $R$ be a clean ring. Then idempotents lift modulo every ideal $I$ of $R$.
\begin{proof}
Let $a\in R$ be such that $a^2-a\in I$. By assumption  $a=e+u$, for some idempotent $e$ and a unit $u$ of $R$. Then $a-u(1-e)u^{-1}= e+ueu^{-1} +u-1=(a^2-a)u^{-1}\in I  $. This shows that $a$ lifts to $u(1-e)u^{-1}$.
\end{proof}
\end{lemma}

For an element $a\in R$ the canonical image of $a$ in the factor ring $R/I$ will be denoted by $\bar a$.

\begin{definition} Let $I$ be an ideal of a ring $R$.
We say that idempotents lift   up to   conjugation modulo $I$    if:\\
(i)    idempotents lift modulo $I$;\\
  (ii) if $e,f\in R$ are idempotents such that $\bar e= \bar f$, then $e$ and $f$ are conjugate in $R$.
\end{definition}
Clearly if idempotents lift uniquely modulo $I$, then they also lift up to conjugation. It is known that if $I$ is a nil ideal of $R$, then idempotents lift modulo $I$. The following lemma, which is a direct consequence of Proposition \ref{Cor idempotents}, says that in this case idempotents lift up to conjugation.
\begin{lemma}
 Let $I$ be an ideal of  $R$ contained in $J(R)$ such that idempotents lift modulo $I$. Then  idempotents lift up to conjugation modulo $I$.
\end{lemma}

 \begin{lemma}\label{lifting unit}
  Let $I$ be an ideal of $R$ contained in $J(R)$. Then:
 \begin{enumerate}\label{lifting conjugation}
   \item[(1)]  Let $a\in R$. If $\bar a=a+I\in R/I$ is invertible in $R/I$, then $a$ is invertible in $R$. In particular, $a+s $ is invertible, for any $s\in I$.
   \item[(2)]   Suppose   that idempotents lift modulo $I$.  If $e,f$ are idempotents of $R$ such that $\bar e, \bar f$ are conjugate in $R/I$, then  $e,f$ are conjugate in $R$.
 \end{enumerate}

 \end{lemma}
 \begin{proof}
 (1) Suppose  $\bar a\in R/I$ is invertible in $R/I$. Then there exist  $b\in R$ and $s,t\in I$ such that $ab=1+s, ba=1+t$. Since $I\subseteq J(R)$, the elements $1+s, 1+t$ are invertible. This yields the thesis.

 (2) Let $\bar u\in U(R/I)$ be such that $\bar e=\bar u\bar f\bar u ^{-1}$. Then, by $(1)$,  $u$ is invertible in $R$ and
  $e=ufu^{-1}+s$, for some $s\in I\subseteq J(R)$. Now,   Proposition \ref{Cor idempotents} implies that $e$ and $ufu^{-1}$ are conjugate in $R$. Thus $e $ and $f$ are also conjugate.
 \end{proof}
 For any ring $R$, $M_n(R)$ and $UT_n(R)$ will denote the ring of $n $ by $n$  matrices over $R$ and its subring consisting of all upper triangular matrices, respectively.   $\mathbb{F}_2$ will stand for the field $\mathbb{Z}/2\mathbb{Z}$.

 \section{Conjugate (nil) clean rings}

 We begin with the following definition:
 \begin{definition}
 A ring $R$ is conjugate (nil) clean if every element of $R$ is conjugate (nil) clean.
 \end{definition}

Clearly every uniquely (nil) clean ring is conjugate (nil) clean.

It is known (see Lemma 5.5 of \cite{D} and Lemma 4 \cite{NZ}, respectively) that  idempotents in uniquely (nil) clean rings are central. Therefore we have:
\begin{remark}\label{remark conj clean + central idemp}
 The following conditions are equivalent:
\begin{enumerate}
  \item  $R$ is uniquely (nil) clean ring;
  \item  $R$ is conjugate (nil) clean, abelian ring.
\end{enumerate}
\end{remark}
The above suggests   that conjugate (nil) clean rings form a natural extension of the class of uniquely (nil) clean rings.

Theorem 3 of \cite{KKY} states that a matrix ring $M_n(D) $ over a division ring $D$ is nil clean if and only if $D=\mathbb{F}_2$ (in the case $D$ is a field, this result was obtained earlier in \cite{BD}).  With the help of this theorem,    we get the following characterization:

\begin{theorem}\label{nil clean matrices over D}
 Let $D$ be a division ring. Then:
\begin{enumerate}
  \item[(1)]   $M_n(D)$ is conjugate nil clean if and only if $D=\mathbb{F}_2$ and $n\leq 2$,
  \item[(2)]    $M_n(D)$ is   conjugate clean if and only if $D=\mathbb{F}_2$ and $n=1$.
\end{enumerate}
\end{theorem}
\begin{proof} Clearly a division ring $D$ is nil  clean (conjugate clean) if and only if  $D=\mathbb{F}_2$. Therefore   we can restrict our attention to the case when $n\geq 2$.

 $(1)$
Suppose that the ring $ M_n(D)$ is conjugate nil clean. Then it is nil clean and   Theorem 3  \cite{KKY} shows that $D=\mathbb{F}_2$.
 The equation$$\left( \begin{array}{ccc}
                              1  & 1 & 1 \\
                             0 & 1 & 1 \\
                             0 & 0 & 1
                           \end{array}\right)
 = \left( \begin{array}{ccc}
                              1  & 0 & 0 \\
                             0 & 1 & 0 \\
                             0 & 0 & 1
                           \end{array}\right) + \left( \begin{array}{ccc}
                              0  & 1 & 1 \\
                             0 & 0 & 1 \\
                             0 & 0 & 0
                           \end{array}\right) =\left( \begin{array}{ccc}
                              1  & 1 & 1 \\
                             1 & 1 & 1 \\
                             1 & 1 & 1
                           \end{array}\right) + \left( \begin{array}{ccc}
                             0  & 0 & 0 \\
                             1 & 0 & 0 \\
                             1 & 1 & 0
                           \end{array}\right) $$
      implies that $M_n(\mathbb{F}_2)$ is not conjugate nil clean, for any $n\geq 3$. Thus $D=\mathbb{F}_2$ and $n\leq 2$, as required.

 By the same theorem,  $R=M_2(\mathbb{F}_2)$ is nil clean. Let $a=e+l=f+m$ be two nil clean decompositions of $a\in R$. Then $tr(a)=tr(e)=tr(f)$, where $tr(a)$ denotes the trace of the matrix $a$.
If $tr(a)=1$, then both $e$ and $f$ are conjugate to  $ \left(
    \begin{array}{cc}
      1 & 0 \\
      0 & 0 \\
    \end{array}
  \right) $.

Suppose $tr(a)=0$, Then either $e=0$ and $a$ is nilpotent or $e=1$ and $a$ is a unit. This implies that $e=f$, when  $tr(a)=0$. Therefore $M_2(\mathbb{F}_2)$ is conjugate nil clean.

 This completes the proof of (1).

(2) The equation
$$ \left(
    \begin{array}{cc}
      0 & 1 \\
      1 & 0 \\
    \end{array}
  \right)
  =\left(
    \begin{array}{cc}
      0 & 0 \\
      0 & 0 \\
    \end{array}
  \right) +
\left(
    \begin{array}{cc}
      0 & 1 \\
      1 & 0 \\
    \end{array}
  \right)
  =
  \left(
    \begin{array}{cc}
      1 & 0 \\
      0 & 0 \\
    \end{array}
  \right)
  + \left(
      \begin{array}{rc}
        -1 & 1 \\
        1 & 0 \\
      \end{array}
    \right)
$$ yields   that, for any ring $R$ and $n\geq 2$,  the matrix ring $M_n(R)$ is not conjugate   clean. This and the remark from the beginning of the proof imply $(2)$.
\end{proof}
  The following corollary is a direct consequence of the above theorem and its proof. It shows, in particular, that  a conjugate nil clean  ring does not have to be  conjugate clean. It is worth to mention  that every uniquely nil clean ring is uniquely clean (cf. Theorem 5.9\cite{D}).
\begin{corollary}\label{size of matrces}
(1) The ring $M_2(\mathbb{F}_2)$ is conjugate nil clean and it is neither  conjugate clean nor uniquely nil clean;\\
(2) Let $R$ be a ring of characteristic 2 and $n\geq 3$. Then    $M_n(R)$ is not conjugate nil clean; \\
(3) For any ring $R$ and $n\geq 2$,  $M_n(R)$ is never conjugate clean.
\end{corollary}
We will see in Proposition \ref{there is no conj nil clean matrices} that the assumption about the characteristic of $R$ in the above corollary can be removed.

Let us record the following property, its easy proof is left as an exercise.
\begin{prop}\label{prop product}
  The   product $R_1\times \ldots\times  R_n $ is conjugate (nil) clean if and only if  all rings $R_i$, $1\leq i\leq n$, are such.
\end{prop}

\begin{prop}\label{prop matrices over Boolean}
Let $R$ be a Boolean ring. Then:
 \begin{enumerate}
   \item[(1)] $M_n(R)$ is nil clean, for any $n\geq 1$;
   \item[(2)] $M_n(R)$ is conjugate nil clean if and only if $n\leq 2$.

 \end{enumerate}
\end{prop}
\begin{proof}
 The statement (1) is exactly Corollary 6 of \cite{BD}.

For proving (2) we will extend arguments used in \cite{BD}. By (1), the ring  $M_2(R)$  is nil clean. We claim that it is conjugate nil clean.  Let $a=e+l=f+m\in M_2(R)$ be nil clean decompositions of $a$. Let $S$ be a the subring of $R$ generated by all entries of matrices appearing in the above equations. Then $S   $ is a finite  Boolean ring, so it is isomorphic to finite direct  product of copies of $\mathbb{F}_2$. Hence $M_n(S) $ is isomorphic to a finite product of copies of $M_2(\mathbb{F}_2)$ and
Theorem \ref{nil clean matrices over D} and Proposition \ref{prop product} yield that $M_2(S)$ is a  conjugate nil clean  ring.  Therefore  $e$ and $f$ are conjugate in $M_2(S)\subseteq M_2(R)$. Notice that every finite subring of a Boolean ring is a direct summand, so $R=S\oplus T$, for some subring $T$ of $R$. Then $M_2(R)=M_2(S)\oplus M_2(T)$ and it is clear  that $e$ and $f$ are conjugate  in $M_2(R)$, i.e. $M_2(R)$ is conjugate nil clean.

The reverse implication is given by Corollary \ref{size of matrces}(2), as Boolean rings are  of characteristic 2.
\end{proof}

\begin{prop} \label{thm conjugate clean}
 Let $I$ be an ideal of $R$ contained in $J(R)$ such that idempotents lift modulo $I$. Then $R$ is conjugate clean if and only if $R/I$ is conjugate clean.
\end{prop}
\begin{proof}
 Suppose $R/I$ is conjugate clean. Let $a\in R$ and $\bar a=\bar e +\bar v$ be clean decomposition of $\bar a$ in $R/I$. Since idempotents lift modulo $I$, Lemma \ref{lifting unit} implies that we may assume that $e$ is an idempotent and  $a=e+v+s$, where $v $ is a unit of $R$ and $s\in I\subseteq J(R)$. Then $v+s\in U(R)$ by Lemma  \ref{lifting unit} again, i.e. $a$ is a clean element of $R$. If $a=e+s=f+t$ are two clean presentations of $a$  then, by assumption $\bar e, \bar f$ are conjugate in $R/I$ and Lemma \ref{lifting conjugation} implies that $e,f$ are conjugate in $R$.

  Suppose now that $R$ is conjugate clean. Then clearly $R/I$ is clean. By assumptions imposed on $I$,   both idempotents and units lift modulo $I$. Using this,  it is easy to see    that the ring $R/I$ is conjugate clean.
\end{proof}

Suppose that  $T$ is an over ring of a ring $R$ such that $T=R\oplus I$, for some ideal $I$ of $T$. In this situation it is clear that idempotents lift modulo $I$. The results, which were obtained up to now,    give the following corollaries.

\begin{corollary}\label{constraction of cojugate clean} Suppose $T=R\oplus I$, where $I$ is an ideal of $T$ such that $J(I)=I$. Then:
 \begin{enumerate}
   \item[(1)]  $T$ is conjugate clean if and only if the ring $R$ is conjugate clean;
   \item[(2)] Suppose $\an{R}{I}=0$. Then  $T$ is uniquely  clean if and only if the ring $R$ is uniquely clean and  all its idempotents    commute with elements of $I$.
 \end{enumerate}
\begin{proof}
 The first statement is a direct consequence of Proposition \ref{thm conjugate clean}. The second one follows from Theorem \ref{thm1 central idempotents}, Remark \ref{remark conj clean + central idemp} and (1). \end{proof}\end{corollary}
 \begin{corollary}\label{cor T+I}
\begin{enumerate}
\item[(1)] Let $UT_n(R)$ denote the ring of all $n$ by $n$ upper triangular matrices over $R$. Then $UT_n(R)$ is conjugate clean if and only if  $R$ is such;
\item[(2)] Let $\sigma$ be an endomorphism of a ring $R$ and $n\geq 2$.  If $T$ denotes one of the rings $R[[x;\sigma]]$, $R[x;\sigma]/(x^n)$, then:\\
 (i) $T$ is conjugate clean if and only if  $R$ is conjugate clean;\\
(ii) $T$ is uniquely clean if and only if  $R$ is uniquely clean and $\sigma(e)=e$, for every idempotent $e$ of $R$.
\end{enumerate}
\end{corollary}
\begin{proof}
Corollary \ref{constraction of cojugate clean} and Proposition \ref{prop product} give the first statement. The statement (2) is a consequence of Corollaries \ref{constraction of cojugate clean} and  \ref{idempotents in power series}.
\end{proof}

Using the above, we can easily construct rings which are conjugate clean but are not uniquely clean.

\begin{example}      Let  $R$ be a uniquely clean ring.  The following rings are conjugate clean but they are not uniquely clean:  $UT_n(R)$ and    $R[[x;\sigma]]$, $R[x;\sigma]/(x^n)$, where $n\ge 2$ and  $\sigma$ is an endomorphism of $R$ such that there exists an idempotent  $e$ of $R$ with $\sigma(e)\ne e$.
                    \end{example}

The following theorem   offers   characterizations of conjugate clean rings. The statement (3) gives a way of constructing new  conjugate clean rings from  a  given conjugate clean ring. Clearly this construction generalizes the one from Corollary \ref{constraction of cojugate clean}.
\begin{theorem}\label{Cor description of conjugate clean}
The following conditions are equivalent:
 \begin{enumerate}
   \item[(1)] $R$ is conjugate clean;
   \item[(2)] $R/J(R)$ is conjugate clean and idempotents lift modulo $J(R)$.
\item[(3)]  There exist a conjugate clean subring $A$ of $R$ and a Jacobson radical ideal $I$ of $R$ such that:\\
    (i) $R=A+I$;\\
(ii)   $U(A)=U(R)\cap A$;\\
    (iii) every idempotent of $R$ is of the form $e+x$, for some $e=e^2\in A$ and $x\in I$.
 \end{enumerate}
\end{theorem}
\begin{proof}
 The  equivalence $(1)\Leftrightarrow (2)$ is a direct consequence of Proposition \ref{thm conjugate clean} and Lemma \ref{lemma 17 NZ}.

Taking $A=R$ and $I=0$, one gets $(1)\Rightarrow (3)$.

$(3)\Rightarrow (1)$ Let $A$ and $I$ be as in $(3)$ and $r\in R$. Then $r=a+x$, for some $a\in A$ and $x\in I$. Since $A$ is   clean, there exist $e=e^2\in A$ and $u\in U(A)\subseteq U(R)$ such that $a=e+u$. By assumption $I\subseteq J(R)$, so $u+x$ is invertible in $R$ and $r=e+(u+x)$ is a clean decomposition of $r$. This shows that $R$ is a clean ring.

Let $r=e+u$ be a clean decomposition of $r\in R$. By $(iii)$, $e=e_0+x$, for an idempotent $e_0 \in A$ and $x\in I\subseteq J(R)$. Proposition \ref{Cor idempotents} shows that $e$ and $e_0$ are conjugate. Moreover $r=e_0+(u+x)$ is a clean decomposition of $r$.
Let us consider two clean decompositions of $r$, say $r=e+u=f+v$. By the above, up to conjugation of idempotents, we may assume that $e,f\in A$, $u,v\in U(R)$. Let $u_0\in A$ and $x\in I\subseteq J(R)$ be such that $u=u_0+x$. Then $u-x=u_0\in U(R)\cap A=U(A)$.
Similarly, there exist $v_0\in U(A)$ and $y\in I$ such that  $v=v_0+y$. Then $r-x=f+v_0+y-x=e+u_0\in A$.
In particular, we get $y-x\in A\cap I$ and $v_0+y-x\in U(R)\cap A=U(A)$. Therefore, as $A$ is conjugate clean and $ f+(v_0+y-x)=e+u_0 $ are clean decompositions of $r-x\in A$, $e$ and $f$ are conjugate. This shows that $R$ is conjugate clean.
\end{proof}

 Notice that if an element $a\in R$ can be written in a form $a=e+t$, where $e=e^2$ and $t\in J(R)$, then $a=(1-e)+ (2e-1)+t$. Since $t\in J(R)$ and $2e-1$ is a unit, the above  equation shows that $a$ is a clean element. Therefore rings in which every element can be presented as a sum of an idempotent and an element from $J(R)$ form a natural proper subclass of clean rings. We call such rings $J$-clean rings.

Making use of Lemma \ref{lemma 17 NZ}, one can easily  check that  $R$ is $J$-clean if and only if $R/J(R)$ is a Boolean ring  and idempotents lift modulo $J(R)$.   Uniquely clean rings were characterized in \cite{NZ}, as rings  $R$  such that $R/J(R)$ is Boolean and idempotents lift uniquely modulo $J(R)$. Therefore, the class of  uniquely clean rings is contained in the class of $J$-clean ring.  The inclusion is strict, since uniquely clean rings are abelian. Notice that, by  Theorem \ref{Cor description of conjugate clean} we get:
 \begin{corollary}
  Every $J$-clean ring is conjugate clean.
 \end{corollary}

The remaining part of this section is focused on properties of conjugate nil clean rings.

Let $I$ be a nil ideal of a ring $R$. Then $h\in R$ is nilpotent if and only it $\bar h\in R/I$ is such. In particular,  if $h\in R$ is nilpotent,  then all elements from the coset $h+I$ are also nilpotent.   Using this observation and arguments similar to that of Proposition \ref{thm conjugate clean} and Theorem \ref{Cor description of conjugate clean} we can  prove the following:

\begin{prop}\label{thm conjugate nil clean}
 Let $I$ be a nil ideal of $R$. The following conditions are equivalent:
 \begin{enumerate}
   \item[(1)] $R$ is conjugate nil clean;
   \item[(2)] $R/I$ is conjugate nil clean;
\item[(3)] There exists a conjugate nil clean subring $A$ of $R$ such that:\\
(i) $R=A+I$;\\
(ii)  every idempotent of $R$ is of the form $e+x$, where $e=e^2\in A$ and $x\in I$.
 \end{enumerate}
\end{prop}
\begin{proof}
 $(1)\Rightarrow (2)$   Suppose   that $R$ is conjugate nil clean. Then clearly $R/I$ is nil clean. Let $\bar a=\bar e +\bar l=\bar f+\bar m$ be nil clean decompositions of $\bar a$. Since $I$ is nil, idempotents lift modulo $I$, so we may assume that $e,f$ are idempotents of $R$. Clearly  elements $l,m$ are nilpotent as $\bar l, \bar m$ are such. Hence, in $R$, we can write $a=e+l+s=f+m+t$ for some suitable $s,t\in I$. Then the elements $l+s$ and $m+t$ are also nilpotent and the fact that  $R$ is conjugate clean yields that $e$ and $f$ are conjugate in $R$. Thus  $\bar e$ and $\bar f$ are conjugate in $R/I$, i.e. $R/I$ is conjugate clean.

 $(2)\Rightarrow (1)$ Suppose $R/I$ is conjugate nil clean. Let $a\in R$. Then, by assumption, we have nil clean decomposition $\bar a=\bar e +\bar l$ in $R/I$. Since idempotents lift modulo $I$, we may assume that $e$ is an idempotent and  $a=e+l+s$, where $l+s$ is nilpotent as $l$ is such and $s\in I$. This shows that $a$ is a nil clean element of $R$. If $a=e+l=f+m$ are two nil clean presentations of $a$  then, by assumption, $\bar e, \bar f$ are conjugate in $R/I$ and Lemma \ref{lifting conjugation} implies that $e,f$ are conjugate in $R$.

$(3)\Rightarrow (1)$ Let $R=A+I$, where $A$ is as in (3). Then  every   $r\in R$ can be presented in a form
$r=a+x=e+l+x$, where $a\in A$, $x\in I$ and $a=e+l$ is a nil clean decomposition of $a$ in $A$.  Notice that $l+x$ is a nilpotent element as $l$ is nilpotent and $I$ is a nil ideal. This proves that $R$ is nil clean.

Let $r=e+l$ be a nil clean decomposition of $r\in R$. Using $(ii)$, we may pick  an idempotent  $e_0\in A$ and $x\in I$ such that $e=e_0+x$. Proposition \ref{Cor idempotents} shows that $e$ and $e_0$ are conjugate in $R$. Moreover $r=e_0+(l+x)$ is a nil clean decomposition of $r$. This means that, considering  nil clean  decompositions $e+l$ of $r\in R$ up to conjugation of idempotents,   we may assume that $e\in A$.

Let $r=e+l=f+m$ be two clean decompositions of $r\in R$. By the above,  we may assume that $e,f\in A$. Let $l_0\in A$ and $x\in I$ be such that $l=l_0+x$. Then $l_0=l-x$ is nilpotent, as $l$ is nilpotent and $I$ is nil.
Similarly, there exist a nilpotent element  $m_0\in A$ and $y\in I$ such that  $m=m_0+y$. Then $r-x=f+m_0+y-x=e+l_0\in A$.
In particular, we get $y-x\in A\cap I$ and $m_0+y-x\in A$ is nilpotent. Therefore, as $A$ is conjugate nil clean and $ f+(m_0+y-x)=e+l_0 $ are two clean decompositions of $r-x$ in $A$, $e$ and $f$ are conjugate in $A$ so in $R$ as well.  This shows that $R$ is conjugate nil clean.

Taking $A=R$  one gets $(1)\Rightarrow (3)$.
\end{proof}
 The first application of the above proposition requires the following observation.
\begin{lemma} \label{lem. reduction to F_2 algebras}
 Let $R$ be a nil clean ring. Then $2R$ is a nilpotent ideal of $R$. In particular $R/2R$ has a structure of $\mathbb{F}_2$-algebra.
\end{lemma}
\begin{proof}
  Let
 $2=e+l$ be nil clean decomposition of $2$. Then $e+el=2e=e+le$. Hence $le=e=el$ and, as $l$ is nilpotent, we obtain $e=0$, i.e. $2=l$ is nilpotent.  Thus $ 2R$ is a nilpotent ideal of $R$ and the ring $\hat R=R/2R$ has a structure of an $\mathbb{F}_2$-algebra, as required.
\end{proof}
\begin{prop}\label{there is no conj nil clean matrices}
 For any ring $R$ and $n\ge 3$, the ring $M_n(R)$ is not conjugate nil clean.
\end{prop}
\begin{proof}
Let $n\ge 3$. Assume that $R$ is a ring such that $M_n(R)$ is conjugate nil clean. By Lemma \ref{lem. reduction to F_2 algebras}, $I=2M_n(R)=M_n(2R)$ is a nilpotent ideal of $M_n(R)$. Thus, Proposition \ref{thm conjugate nil clean} implies that  $M_n(R)/I\simeq M_n(R/2R)$ is conjugate nil clean. Corollary \ref{size of matrces}(2) shows that this is impossible, as  $R/2R$ is of characteristic 2. Thus   such $R$ can not exist.
\end{proof}
It is known (Proposition 3.16 \cite{D}) that if $R$ is nil clean then $J(R)$ is nil. In particular $R$ is nil clean if and only if $R/J(R)$ is nil clean and $J(R)$ is nil. Proposition \ref{thm conjugate nil clean} gives the following characterization of conjugate nil clean rings.
\begin{corollary}\label{cor. conjugate nil clean}
 For a ring $R$, the following conditions are equivalent:
 \begin{enumerate}\label{cor. conj. nil clean characterization}
   \item[(1)]    $R$ is conjugate nil clean;
   \item[(2)]   $J(R)$ is nil and  $R/J(R)$ is conjugate nil clean;
\item[(3)]   J(R) is  nil   and there exist a conjugate nil clean subring $A$ of $R$ such that:\\
    (i) $R=A+J(R)$;\\
(ii) every idempotent of $R$ is of the form $e+x$, for some $e=e^2\in A$ and $x\in J(R)$.
 \end{enumerate}
\end{corollary}

Proposition 3.18 \cite{D} states that if $R$ is a nil clean, abelian ring,    then $J(R)$ contains all nilpotent elements of $R$. Using Corollary \ref{cor. conjugate nil clean} and Remark \ref{remark conj clean + central idemp} one can easily recover (cf. Theorem 5.9\cite{D}) the following characterization of uniquely nil clean rings:
\begin{corollary} \label{nil clean char} For a ring $R$, the following statements are equivalent:
\begin{enumerate}
  \item[(1)]  $R$ is uniquely nil clean;
  \item[(2)]  $R/J(R)$ is Boolean, $J(R)$ is nil and idempotents  lift uniquely modulo $J(R)$.
 \end{enumerate}

\end{corollary}
 We will use the above characterization in the proof of Theorem \ref{koethe}.

Applying Proposition \ref{thm conjugate nil clean},  Remark \ref{remark conj clean + central idemp} and Theorem \ref{thm1 central idempotents} we obtain the following:
\begin{corollary}
Suppose $T=R\oplus I$, where $I$ is a nil ideal of $T$. Then:
 \begin{enumerate}
   \item[(1)]  $T$ is conjugate nil clean if and only if the ring $R$ is conjugate nil clean;
   \item[(2)] Suppose $\an{R}{I}=0$. Then $T$ is uniquely nil clean if and only if the ring $R$ is uniquely nil clean and  all idempotents of $R$  commute with elements of $I$.
 \end{enumerate}
\end{corollary}
With the help of the above corollary, similarly as in  Corollary \ref{cor T+I}, we get:
\begin{corollary}
 \begin{enumerate}

\item[(1)] Let $UT_n(R)$ denote the ring of all $n$ by $n$ upper triangular matrices over $R$. Then:\\ (i) $UT_n(R)$ is conjugate nil clean if and only if  $R$ is such;\\
(ii) $UT_n(R)$ is not uniquely nil clean when $n\ge 2$.

\item[(2)] Let $\sigma$ be an endomorphism of a ring $R$ and $n\geq 2$. Then:\\
 (i)   $R[x;\sigma]/(x^n)$ is conjugate nil clean if and only if  $R$ is conjugate nil clean;\\
(ii) $R[x;\sigma]/(x^n)$ is uniquely  nil clean if and only if  $R$ is uniquely clean and $\sigma(e)=e$, for every idempotent $e$ of $R$;
    \end{enumerate}
\end{corollary}

Let us notice that Corollary  \ref{nil clean char} implies that if $R$ is a uniquely nil clean ring, then the set of all nilpotent elements $N(R)$ of $R$ is equal to  $J(R)$. In particular, $N(R)$ is an ideal of $R$  in this case.
\begin{prop} Let $R$ be a ring such that the set $N(R)$ is an ideal of $R$. The following conditions are equivalent:

\begin{enumerate}\label{N is an ideal nli clean}
\item[(1)]  $R$ is   nil clean;
  \item[(2)]  $R$ is conjugate nil clean;

\item[(3)] $R/J(R)$ is a Boolean ring and $J(R)$ is nil.
\end{enumerate}
If one of the above  equivalent conditions holds, then $R$ is conjugate clean.
    \end{prop}
\begin{proof}
$(1)\Rightarrow (2)$ Suppose $R$ is nil clean. Let $a=e+l=f+m$ be two nil clean decomposition of $a\in R$. Then $e-f=m-l\in N(R)$. In particular $e-f$ is nilpotent and Proposition \ref{Cor idempotents} shows that the idempotents $e$ and $f$ are conjugate, i.e. $R$ is conjugate nil clean.

$(2)\Rightarrow (3)$ Suppose $R$ conjugate nil clean.  Then, by Corollary \ref{cor. conj. nil clean characterization}, $J(R)$ is nil and $R/J(R)$ is conjugate nil clean. Since $N(R)$ is an ideal, $N(R)= J(R)$. This means that  $R/J(R)$ is a reduced nil clean ring, so it is a Boolean ring.

The implication $(3)\Rightarrow (1)$ is a direct consequence of the fact that idempotents lift modulo nil ideals. The last statement is a consequence of $(3)$ and Theorem \ref{Cor description of conjugate clean}.
\end{proof}

In the context of the above proposition, let us recall that the ring $M_2(\mathbb{F}_2)$ is conjugate nil clean ring but it is not   conjugate clean. Thus the proposition does not hold without the assumption  made on the set $N(R)$.   Notice also that the power series ring  $\mathbb{F}_2[[x]]$ is a  uniquely clean domain  and it is not nil clean.

When $R$ is a commutative ring, then $N(R)$ is an ideal of $R$ and $R$ is conjugate nil clean if and only if $R$ is uniquely nil clean.  Therefore Proposition \ref{N is an ideal nli clean} gives the following corollary:
\begin{corollary}
 For a commutative ring $R$, the following conditions are equivalent:
\begin{enumerate}
 \item[(1)]  $R$ is   nil clean;
  \item[(2)]  $R$ is uniquely nil clean;

\item[(3)] $R/J(R)$ is a Boolean ring and $J(R)$ is nil.
\end{enumerate}
\end{corollary}
 The equivalence of   $(1)$ and $(3)$ in the above corollary is exactly Corollary 3.20 of \cite{D}.

\section{Nil clean rings and K\"oethe's problem}
K\"{o}the's problem was formulated in 1930, it asks whether a ring $R$    has no nonzero nil one-sided ideals provided $R$ has no nonzero nil ideals.
It is   known (see Theorem 6, \cite{K}) that the
problem has a positive solution if and only if it has positive solution for algebras
over fields.      There are many other
problems in ring theory which are equivalent or related to  it (see \cite{P}). In the    theorem below we indicate  new ones which are associated with nil clean rings.

Diesl  in \cite{D} formulated a few questions on  nil clean elements and rings. In particular, he posed a question (Question 3 \cite{D})  whether a matrix ring  $M_n(R)$ over a nil clean ring $R$ has to be nil clean.
 We show that positive answer to the above Diesl 's question is equivalent to positive solution for K\"{o}the's problem in the class of algebras over the field $\mathbb{F}_2$.

\begin{theorem}\label{koethe}
 The following conditions are equivalent:
\begin{enumerate}

\item[(1)] If $R$ is  a nil clean ring, then  $M_n(R) $ is nil clean;
  \item[(2)] If $R$ is a uniquely nil clean ring, then $M_n(R) $ is nil clean;
 \item[(3)] If $R$ is a uniquely nil clean ring, then $M_2(R) $ is nil clean;
 \item[(4)] If $R$ is a uniquely nil clean ring, then $M_2(R) $ is conjugate nil clean;
 \item[(5)] If $A$ is a nil algebra over $\mathbb{F}_2$, then $M_n(A)$ is nil;
   \item[(6)] K\"{o}the's problem has positive solution in the class of $\mathbb{F}_2$-algebras.

\end{enumerate}
\end{theorem}
\begin{proof}

The implication $(1)\Rightarrow (2)\Rightarrow (3)$ is a tautology.

$(3)\Rightarrow (4)$ Let $R$ be a uniquely clean ring. Then,  by Corollary \ref{nil clean char}, $\bar R=R/J(R)$ is a Boolean ring and Proposition \ref{prop matrices over Boolean}(2) implies that the matrix ring $M_2(\bar R)$ is conjugate nil clean.

By assumption, the ring $T=M_2(R) $ is nil clean. In particular, by Proposition 3.16 \cite{D},  $J(T)$ is nil.  Moreover  $T/J(T)=M_2(R)/J(M_2(R))\simeq M_2(\bar R) $ is conjugate nil clean.  Now, Corollary \ref{cor. conjugate nil clean} shows that $T=M_2(R) $ is conjugate nil clean, i.e. $(4)$ holds.

 $(4)\Rightarrow (5)$ Let $A$ be a nil algebra over the field $\mathbb{F}_2$ and $A^*$ denote  the $\mathbb{F}_2 $-algebra
with unity adjoined with the help of $ \mathbb{F}_2$ to $A$. Then $ J(A^*)=A$ and   $A^*/J(A^*)= \mathbb{F}_2$. Since $A^*=A\dot{\cup} (1+A)$, the only idempotents of $A^*$ are 0,1. Thus, by Corollary \ref{nil clean char}, $A^*$ is uniquely nil clean. Therefore, by $(4)$, $M_2(A^*)$ is conjugate nil  clean  and Corollary \ref{cor. conj. nil clean characterization} implies that $J(M_2(A^*))=M_2(J(A^*))=M_2(A)$ is nil. Therefore, we have shown  that for any nil algebra $A$, the $2\times 2$ matrix algebra $M_2(A)$ is also nil. It is known and easy that in this case $M_n(A) $ is nil for any $n\geq 2$.

The equivalence of (5) and (6) is known (cf. \cite{K}).

 $(5)\Rightarrow (1)$    Let $R$ be uniquely nil clean ring. Thus, by Corollary \ref{nil clean char}, $J(R)$ is nil, $R/J(R)$ is Boolean.

Let $I=2R$ and $\hat{R}=R/2R$. Then, by Lemma \ref{lem. reduction to F_2 algebras}, $I $ is a nilpotent ideal of $R$ and  $\hat R$ has a structure of $\mathbb{F}_2$-algebra.  Moreover
   $J(\hat R)=J(R)/I$ is nil and $B=\hat R/J(\hat R)\simeq R/J $ is Boolean. Then, using the statement $(5)$ and Proposition \ref{prop matrices over Boolean}(1),  we obtain that  $J=M_n(J(\hat R))$ is nil and $M_n(\hat R)/J=M_n(B)$ is nil clean, respectively.  Hence, by Corollary 3.17\cite{D} we obtain that $M_n(\hat R)$ is nil clean. Then also $M_n(R)$ is nil clean as $M_n(\hat R)=M_n(R)/M_n(I)$ and $M_n(I)$ is a nilpotent ideal of $M_n(R)$.
\end{proof}

Let us notice that in the proof of the implication $(4)\Rightarrow (5)$, the property $(4)$ was used only for $\mathbb{F}_2$-algebras. This means that in Theorem \ref{koethe} we can add new equivalent statements replacing rings by $\mathbb{F}_2$-algebras.
In particular, Diesl 's question is equivalent to the question whether $M_2(R)$ is conjugate clean for any uniquely clean $\mathbb{F}_2$-algebra.  In this context, let us notice that the ring $R=M_2(\mathbb{F}_2)$ is conjugate nil clean however, by
Corollary \ref{size of matrces},  $M_2(R)=M_4(\mathbb{F}_2)$ is not conjugate nil clean.

Let us mentioned at the end that  relations between K\"{o}the's problem and properties of clean elements were investigated in \cite{KLM2}. In particular, it was proved that K\"{o}the's problem has positive solution if and only if the set of clean elements of the polynomial ring $R[x]$ forms a subring, for any clean ring $R$ such that $N(R)$ is an ideal of $R$ (cf. Theorem 2.15 \cite{KLM2}).

\bf Acknowledgement. \rm I would like to thank Andr\'{e} Leroy and Jan Okni\'{n}ski for helpful discussions.


\begin{thebibliography}{ABC}
\bibitem{BD}S. Breaz, G. C\v{a}lug\v{a}reanu, P. Danchev, T. Micu, Nil-clean matrix rings, Linear Algebra and its Applications 439 (2013) 3115-3119.

\bibitem{Ch} H. Chen, On uniquely clean rings, Comm.  Algebra 39 (2011), 189-198.

\bibitem{D} A.J. Diesl,  Nil clean rings, J. Algebra 383 (2013), 197-211.
\bibitem{HN} J. Han and W.K. Nicholson,  Extensions of clean rings , Comm. Algebra 29 (2001),
2589-2595.
\bibitem{KLM} P. Kanwar, A. Leroy, J. Matczuk, Idempotents in ring extensions, J.  Algebra 389 (2013), 128-136.
\bibitem{KLM2} P. Kanwar, A. Leroy, J. Matczuk, Clean elements in polynomial rings, Contemporary Math.
  634 (2015), 197-204.
\bibitem{KKY}  M.T. Ko\c{s}an,  T.K. Lee, Y. Zhou, When is every matrix over a division ring a sum of
an idempotent and a nilpotent?, Linear Algebra and its Applications 450 (2014), 7-12.
\bibitem{K} J. Krempa, Logical connections between some open problems concerning nil rings, Fund.
Math. 76 (1972), no. 2, 121-130.

\bibitem{N1} W.K. Nicholson, Lifting idempotents and exchange rings, Trans. Amer. Math. Soc. 229 (1977) 269-278.


\bibitem{N} W.K. Nicholson,  Strongly clean rings and Fitting's lemma, Comm. Algebra 27 (1999),
3583-3592.

\bibitem{NZ}  W.K. Nicholson, Y. Zhou, Rings in which elements are uniquely the sum
of an idempotent and a unit, Glasgow Math. J. 46 (2004) 227-236.

\bibitem{NZ2}  W.K. Nicholson, Y. Zhou, Clean Rings: A Survey, Advances in ring theory: Proc.  4th China-Japan-Korea International Symposium, World Sci. Publ., Hackensack, NJ, (2005) , 181-198.

\bibitem{P} E.R. Puczylowski, Questions related to Koethe's nil ideal problem,
 Contemporary Math. 419 (2006), 269-283.


\bibitem{WC} Z. Wang, J. Chen, On two open problems about strongly clean rings ,Bull. Aust. Math. Soc.
  70 (2004),  279-282.

\end{thebibliography}
\end{document}